\documentclass[oneside]{amsart}

\usepackage[letterpaper,body={14.0cm,22.0cm}, mag=1000]{geometry}
\usepackage{amssymb}
\usepackage{amsthm}
\usepackage{amscd}
\usepackage{enumitem}
\usepackage[pdftex]{graphicx}
\usepackage{caption}
\usepackage{amsmath}
\usepackage{tikz}
\usepackage[all,cmtip]{xy}

\numberwithin{equation}{section}
\theoremstyle{plain}

\newtheorem{thm}{Theorem}[section]
 \newtheorem{cor}[thm]{Corollary}
 \newtheorem{lemma}[thm]{Lemma}
\newtheorem{prop}[thm]{Proposition}

\newtheorem*{thma}{Theorem A}
\newtheorem*{thmb}{Theorem B}

\theoremstyle{definition}

\newcommand{\dlabel}[1]{\ifmmode \text{\ttfamily \upshape [#1] } \else
{\ttfamily \upshape [#1] }\fi \label{#1}}

\newcommand{\Ho}{\operatorname{H} }
\newcommand{\M}{\operatorname{M} }

\newcommand{\Z}{\operatorname{Z} }

\newcommand{\id}{\operatorname{Id} }

\newcommand{\gen}[1]{\left < #1 \right >}

\newcommand{\res}{\operatorname{res} }

\newcommand{\Hom}{\operatorname{Hom} }
\newcommand{\Ker}{\operatorname{Ker} }

\newcommand{\im}{\operatorname{Im} }
\newcommand{\tra}{\operatorname{tra} }

\begin{document}

\title{The Schur multiplier of central product of groups}

\author{Sumana Hatui}
\address{School of Mathematics, Harish-Chandra Research Institute, Chhatnag Road, Jhunsi, Allahabad 211019, India \& Homi Bhabha National Institute, Training School Complex, Anushakti Nagar, Mumbai 400085, India}
 \email{sumanahatui@hri.res.in}

\author{L. R. Vermani}
\address{961, Sector 7, Urban Estate, Kuruskhetra, India}
 \email{lrver@yahoo.com}

 \author{Manoj K. Yadav}
\address{School of Mathematics, Harish-Chandra Research Institute, Chhatnag Road, Jhunsi, Allahabad 211019, INDIA \& Homi Bhabha National Institute, Training School Complex, Anushakti Nagar, Mumbai 400085, India}
 \email{myadav@hri.res.in}

\thanks{The second author would like to thank Harish-Chandra Research Institute, Allahabad for the excellent hospitality provided to him couple of times during which a major part of this work was done.}
\subjclass[2010]{20J06}
\keywords{Second cohomology group, Schur multiplier, central product}

\begin{abstract}
Let $G$ be a central product of two groups $H$ and $K$. We study second cohomology group of $G$, having coefficients in a  divisible abelian group $D$ with trivial $G$-action,  in terms of the second cohomology groups of certain quotients of $H$ and $K$. In particular, for $D =  \mathbb{C}^{*}$, some of our results provide a refinement of results from [Some groups with non-trivial multiplicators, Math. Z. {\bf 120 } (1971), 307-308] and [On the Schur multiplicator of a central quotient of a direct product of groups,  J. Pure Appl. Algebra {\bf 3} (1973), 73-82].
\end{abstract}

\maketitle 

\section{Introduction}

The Schur multiplier $\M(G)$ of a given group $G$, introduced by Schur in 1904 \cite{IS04},  is the second cohomology group $\Ho^2(G, \mathbb{C}^{*})$  of $G$ with coefficients in $\mathbb{C}^{*}$.
Let  a finite group $G$ be the  direct product of two groups $H$ and $K$. Then the formulation of the Schur multiplier of $G$  in terms of the Schur multipliers of $H$ and $K$ was given by Schur himself \cite{IS07}. Such a formulation, when $G$ is a semidirect product of  groups $H$ and $K$, was given by Tahara \cite{KT}.
 
We say that $G$ is an internal central product of   normal subgroups $H$ and $K$ amalgamating $A$ if $G=HK$ with $A=H \cap K$  and $[H,K]=1$.
Let $H$, $K$ be two groups with  isomorphic subgroups $A \leq \Z(H)$, $B \leq \Z(K)$ under an  isomorphism $\phi: A \rightarrow B$. Consider the normal subgroup $U = \{( a, \phi(a)^{-1}) \mid a \in A\}$. 
Then the group $G :=(H \times K)/U$ is called the external central product of $H$ and $K$ amalgamating $A$ and $B$ via $\phi$. The external central product $G$ can be viewed as an internal central product of  the images of $H \times 1$ and $1 \times K$ in $G$. For this reason, we do not differentiate between external and internal central products, and consider only internal ones.

Let $G$ be a finite group which is a central product of  subgroups $H$ and $K$ amalgamating $A$.  Wiegold \cite{JW} proved that $\M(G)$ contains a subgroup isomorphic to $H/A \otimes K/A$, the abelian tensor product of $H/A$ and $K/A$. A generalization of this result for an arbitrary central quotient of direct product of two arbitrary groups was considered  in \cite{EHS}.

Recall that $\Ho^2(G, A)$ denotes the second cohomology group of a group $G$ with coefficients in a $G$-module $A$. We are particularly interested in the case  when $A$ is a trivial $G$-module and is divisible. We reserve  $D$ for such a module. Throughout the paper, unless said otherwise explicitly, $G$ is always a  central product of its normal subgroups $H$ and $K$ with $A=H \cap K$. In this paper we study $\Ho^2(G, D)$,
  in terms of the second cohomology groups of certain quotients of $H$ and $K$ with coefficients in $D$.
  Set  $Z = H' \cap K'$, where $X'$ denotes the commutator subgroup of a group $X$. The following result provides a reduction to the case when $Z = 1$.

\begin{thma}\label{thma}
Let $B$ be a subgroup of $G$ such that $B \le Z$. Then $$\Ho^2(G, D) \cong \Ho^2(G/B, D)/N,$$
where $N \cong \Hom (B, D)$.
\end{thma}

This result is very useful for computational purposes when $G$ is finite and $M(G/B)$ is known. Just to elaborate, we immediately get the following result for finite extraspecial $p$-groups proved in \cite[Corollary 3.2]{BE}.

\begin{cor}
Let $G$ be an extra-special $p$-group of order $p^{2n+1}$, $n \geq 2$. Then  $M(G)$ is an elementary abelian $p$-group of order $p^{2n^2-n-1}.$
\end{cor}

By the tensor product $G_1 \otimes G_2$ of two groups $G_1$ and $G_2$, we always mean the abelian tensor product, i.e., $G_1/G_1' \otimes G_2/G_2'$.  Our next result is the following:

\begin{thmb}\label{thmb}
Let $L \cong \Hom \big((A\cap H')/Z, D \big)$,  $M \cong \Hom \big((A\cap K')/Z, D \big)$ and $N \cong \Hom(Z, D)$.
 Then the following statements hold true:

(i)  $\big(\Ho^2(H/A, D)/L \oplus \Ho^2(K/A,D)/M\big)/N \oplus \Hom(H/A \otimes K/A, D) $ embeds in $\Ho^2(G, D)$;

(ii) $\Ho^2(G, D)$ embeds in  $\big(\Ho^2(H/Z, D) \oplus \Ho^2(K/Z,D)\big)/N \oplus \Hom(H \otimes K, D).$
\end{thmb}

 In particular, for $D =  \mathbb{C}^{*}$, assertion (i)  of Theorem B provides a refinement of results from  \cite{EHS} and  \cite{JW}.
 
On the way to proving these results, we obtain some commutative diagrams and  exact sequences which might be of independent mathematical interest. Although our techniques are mostly cohomological, free presentation also shows up occasionally. In Section 2 we recall some known results and definitions, and establish  a basic commutative diagram, which we refer to several times in what follows.  Proofs of Theorems A and B are presented in Section 3. The final section contains several examples exhibiting various situations in which we determine whether or not  any embedding in Theorem B actually becomes isomorphism.

\vspace{.1in}

\section{Notations and preliminaries}

\vspace{.05in}

Let $F/R$ be a free presentation of $X$ and $N$ be a normal subgroup of $X$. Let $S/R$ be the induced free presentation of $N$ for some subgroup $S$ of $F$.  The following crucial result then  follows from \cite[Corollary 3.5]{LR}.

\begin{lemma}\label{lemma1}
The inflation homomorphism $\inf : \Ho^2(X/N, D) \to \Ho^2(X, D)$ is surjective if and only if $[F, R] = R \cap [F, S]$.
\end{lemma}

For an arbitrary group $X$ and a subgroup $N$, by $\res^X_N$ we denote the restriction homomorphism from $\Hom(X,D)$ to $\Hom(N,D)$ as well as the restriction homomorphism from $\Ho^2(X,D)$ to $\Ho^2(N,D)$. When the meaning is clear from the context, we write $\res$ for $\res^X_N$.

Let us consider the following central exact sequence for an arbitrary group $X$ and a central subgroup $N$:
\[1 \to N \to X \to X/N \to 1.\]
Then we get the exact sequence 
\[0 \to \Hom(N \cap X', D) \stackrel{\tra}\to \Ho^2(X/N, D)  \stackrel{\inf}\to \Ho^2(X, D) \stackrel{\chi}\to \Ho^2(N, D) \oplus \Hom(X \otimes N, D),\]
where  $\tra:\Hom(N, D) \to \Ho^2(X/N, D)$ is the transgression homomorphism and $\chi = (\res, \psi)$ as defined by  Iwahori and Matsumoto  \cite{IM}.  To be more precise, $\res : \Ho^2(X,D) \to \Ho^2(N,D)$ is the restriction homomorphism and $\psi: \Ho^2(X,D)  \to \Hom(X \otimes N,D)$ is defined as $\psi(\xi)(\bar{x}, n) = f(x,n) - f(n,x)$ for all $\bar{x} = xX' \in  X/X'$ and $n \in N$, where $\xi \in \Ho^2(X,D)$ and  $f$ is a $2$-cocycle representative of $\xi$. 

Define a map
\begin{equation}\label{eqn1}
\theta' : \Ho^2(G,D) \to \Ho^2(H,D) \oplus \Ho^2(K,D) \oplus \Hom(H \otimes K, D)
\end{equation}
 by $\theta'=(\res^G_H, \res^G_K, \nu)$. Here  $\nu : \Ho^2(G,D) \to \Hom(H \otimes K, D)$ is a homomorphism defined as follows.
If $\xi \in \Ho^2(G,D)$ is represented by a $2$-cocycle $f$, then $\nu(\xi)$ is the homomorphism $\bar{f} \in \Hom(H \otimes K, D)$ defined by $\bar{f}(\bar{h} \otimes \bar{k})=f(h,k) - f(k,h)$, where $\bar{h} = hH'$ and $\bar{k} = kK'$. It is now not difficult to see that $\theta'$ is indeed a homomorphism.

Consider the natural homomorphisms $\alpha : AH'/H' \otimes K \to H \otimes K$, $\beta : H \otimes AK'/K' \to H \otimes K$, induced by obvious inclusion maps, and $\lambda :  H \otimes K \to H/A \otimes K/A$,  induced by natural projection. 
We now get the following exact sequence:
\[(AH'/H' \otimes K) \oplus (H \otimes AK'/K') \xrightarrow{\mu_1} H \otimes K \xrightarrow{\lambda} H/A \otimes K/A \to 0,\]
 where $\mu_1(x,y)=\alpha(x)+\beta(y)$.

We have natural epimorphisms $f:H \otimes A \rightarrow  H \otimes AK'/K'$ and $g:A \otimes K \rightarrow  AH'/
H' \otimes K$. Consider the isomorphism $\eta: K \otimes A \to A \otimes K$, which, on the generators, is defined by $\eta(k \otimes a) = (a \otimes k)$. Using this, we have an epimorphism $(f, g \circ \eta) : (H \otimes A) \oplus (K \otimes A) \rightarrow (H \otimes AK'/K') \oplus  (AH'/H' \otimes K)$. Let $\mu =\mu_1 \circ (f, g \circ \eta)$. Then $\im(\mu_1) = \im(\mu)$ and the above exact sequence leads to the  exact sequence:
\[ (H \otimes A) \oplus (K \otimes A) \xrightarrow{\mu} H \otimes K \xrightarrow{\lambda} H/A \otimes K/A \to 0.\]
 This  exact sequence then gives the  exact sequence

\[
 \xymatrix{
0 \;\;\ar[r] &\;\; \Hom(H/A \otimes K/A, D) \;\;\;\ar[r]^{\lambda^*}  &\;\;\;  \Hom(H \otimes K, D) \ar[d]^{\mu^*}\\
&  & \Hom(H \otimes A,D)\oplus \Hom(K \otimes A,D),}
\]
where the homomorphisms $\mu^*$ and $\lambda^*$ are induced by $\mu$ and $\lambda$  respectively. 

Let $\alpha : H/H' \oplus K/K' \rightarrow G/G'$ be the homomorphism induced by the inclusion maps $H \rightarrow G, K \rightarrow G$. Then $\alpha$   is clearly onto. Now $\alpha$ induces an epimorphism $(H/H' \oplus K/K')\otimes A \to G \otimes A$, which in turn induces a monomorphism $\alpha^*:\Hom(G \otimes A, D) \to \Hom(H \otimes A, D) \oplus \Hom(K \otimes A, D)$.
Let
$\bigtriangleup:H^2(A,D) \to H^2(A,D) \oplus H^2(A,D)$ be defined by $\bigtriangleup(\xi)=(\xi,\xi)$ for $\xi \in H^2(A,D)$.
\vspace{.1in}

Set $\bar{G} = G/A$, $\bar{H} = H/A$ and $\bar{K} = K/A$. Let $\xi \in \Ho^2(\bar{G},D)$ and $f$ be a $2$-cocycle representing $\xi$. Recall that 
$$\theta : \Ho^2(\bar{G},D) \to \Ho^2(\bar{H},D) \oplus \Ho^2(\bar{K}, D) \oplus \Hom(\bar{H} \otimes \bar{K}, D)$$
is the isomorphism  defined by 
 $$\theta(\xi) = (\res^{\bar{G}}_{\bar{H}}(\xi), \res^{\bar{G}}_{\bar{K}}(\xi), \nu_1(\xi)),$$ 
 where  $\nu_1 : \Ho^2(\bar{G},D) \to \Hom(\bar{H} \otimes \bar{K}, D)$ is a homomorphism given by  $\nu_1(\xi)(\tilde{h} \otimes \tilde{k})=f(h,k) - f(k,h)$, with $h \in \bar{H}$, $k \in \bar{K}$, $\tilde{h} = h\bar{H}'$ and $\tilde{k} = k\bar{K}'$.
Take $X_1 = \Ho^2(A, D) \oplus \Hom(H \otimes A, D)$, $X_2 = \Ho^2(A, D) \oplus \Hom(K \otimes A, D)$, $X_3 = \Hom(H \otimes A, D) \oplus \Hom(K \otimes A, D)$ and $Y = \Ho^2(A,D) \oplus \Ho^2(A,D)$. We now get the 
 following diagram (Diagram 1) with exact columns. In this diagram,  for want of space, we suppress the use of $D$, i.e., we write $\Hom(X,D)$ as $\Hom(X)$ and $\Ho^2(X,D)$ as $\Ho^2(X)$ for a given group $X$.

\[
 \xymatrix{
    0     \ar[d]&    &0 \ar[d]\\
   \Hom(A \cap G')   \ar[d]_{\tra}  & \xrightarrow{(\res,\res)}  &\Hom(A \cap H') \oplus \Hom(A \cap K') \ar[d]^{(\tra, \tra, 0)}\\
    \Ho^2(G/A)  \ar[d]_{\inf} &  \xrightarrow{\;\;\;\;\;\theta\;\;\;\;\;} &\Ho^2(H/A) \oplus \Ho^2(K/A) \oplus \Hom(H/A \otimes K/A) \ar[d]^{(\inf, \inf, \lambda^*)}\\
  \Ho^2(G)  \ar[d]_{(\res, \psi)}  & \xrightarrow{\;\;\;\;\;\theta'\;\;\;\;\;} &\Ho^2(H) \oplus \Ho^2(K) \oplus \Hom(H \otimes K) \ar[d]^{\big((\res, \psi), (\res, \psi), \mu^*\big)} \\
\Ho^2(A) \oplus \Hom(G \otimes A)  & \ar[rd]_{(\bigtriangleup,\alpha^*, \alpha^*)} &  X_1 \oplus X_2 \oplus X_3 \ar[d]^{\cong} \\
&  &Y \oplus X_3 \oplus X_3.}
\]
\[\text{Diagram 1}\] \\

\begin{lemma}
Diagram 1 is commutative.
\end{lemma}
\begin{proof}
It is a routine check to see that the topmost and middle rectangles are commutative. Observe that $\res^H_A \circ \res^G_H = \res^G_A$ and $\res^K_A \circ \res^G_K = \res^G_A$. It is also clear from the definitions that $(\alpha^*, \alpha^*) \circ \psi = \big((\psi, \psi), \mu^*\big) \circ {\theta}'$. Thus it follows that the bottom part of the diagram is also commutative. \hfill $\Box$

\end{proof}

\section{Proofs}

In this section we present  proofs of the results stated in the introduction. We start with the following result.  

\begin{thm}\label{thm1}
For any central subgroup $B$ of $G$ contained in $H' \cap K' \;(= Z)$, the inflation homomorphism $\inf : \Ho^2(G/B,D) \to \Ho^2(G,D)$ is surjective.
\end{thm}
\begin{proof}
Let $F/R$ be a free presentation of $G$. Then the normal subgroups $H$, $K$ and $B$ can be freely presented as $ S_1/R$, $S_2/R$ and $S/R$ respectively, where $S_1$, $S_2$ and $S$ are  normal subgroups of $F$. Further, $Z \cong (S_1'\cap S_2')R/R$. Note that $F = S_1S_2$, $S \subseteq (S_1' \cap S_2')R$ and $[S_1, S_2] \subseteq R$.

By Lemma \ref{lemma1}, it is enough to prove that $[F,R] = R \cap[F,S]$. Since $[S_1,S,S_2] \subseteq [S_2,S]$, 
we have  
\begin{eqnarray*}
R \cap [F,S] &=& [F,S]=[S_1S_2,S]=[S_1,S][S_1,S,S_2] [S_2,S]\\ &=&[S_1,S][S_2,S].
\end{eqnarray*}
Observe that $[S_1,S] \subseteq [S_1, S_2'R] = [S_1, R][S_1,S_2'][S_1,S_2',R]$. Since both $[S_1,S_2,S_2]$ and $[S_2, S_1, S_2]$ are contained in $[F, R]$, by the three subgroup lemma $[S_1,S_2'] \subseteq [F,R]$. Hence  $[S_1,S] \subseteq [F,R]$. Similarly $[S_2,S] \subseteq [F,R]$. Therefore  $R \cap [F,S] \subseteq [F, R]$. Since $[F, R] \subseteq R \cap [F,S]$, it follows that  $[F,R] = R \cap[F,S]$, and the proof is complete. \hfill $\Box$

\end{proof}

 \noindent{\it Proof of Theorem A.}  It follows from Theorem \ref{thm1} that $\inf$, in the following exact sequence,  is surjective.
\[0 \to \Hom(B, D) \stackrel{\tra}\to \Ho^2(G/B, D)  \stackrel{\inf}\to \Ho^2(G, D).\]
Since $\Hom(B, D) \cong \im(\tra) = \Ker(\inf)$, the proof is complete.   \hfill $\Box$\\

We now mainly concentrate on the homomorphism $\theta'$ defined in \eqref{eqn1}. We start with  the following result  about the kernel of  $\theta'$.

\begin{lemma}\label{lemma2}
$\Ker (\theta')= \{\inf(\eta) \mid \eta \in \theta^{-1} \big(\im (\tra,\tra,0)\big)\}$.
\end{lemma}
\begin{proof}

Let $\xi \in \ker(\theta')$. By the commutativity of the bottommost part of Diagram 1, it follows that $(\bigtriangleup,\alpha^*, \alpha^*)(\res^G_A, \psi)(\xi)=0$. Since $\alpha^*$ is a monomorphism and $\bigtriangleup = (\id, \id)$,  it follows that $(\res^G_A, \psi)(\xi)=0$.  Now the existence of  $\eta \in \Ho^2(G/A,D)$ such that $\xi=\inf(\eta)$ is guaranteed by the  exactness of the left column in Diagram 1. Thus $\Ker (\theta') \subseteq \im \big(\inf:\Ho^2(G/A,D) \rightarrow \Ho^2(G,D) \big)$.

By the commutativity of the middle rectangle of Diagram 1, it follows that 
 \[0 =\theta'(\xi)= \theta'( \inf(\eta)) = \theta' \circ \inf(\eta)=(\inf,\inf,\lambda^*) \circ \theta(\eta).\]
 Again invoking Diagram 1, we get $\theta(\eta) \in \im (\tra,\tra,0)$. Hence $\eta \in \theta^{-1}(\im (\tra,\tra,0))$. That 
 $\theta'(\inf(\eta))=0$ for $\eta \in \theta^{-1} \big(\im (\tra,\tra,0)\big)$ follows from the commutativity of Diagram 1 with the right column exact. This completes the proof.
  $\hfill\square$

\end{proof}

We have an exact sequence 
\[0 \rightarrow H' \cap K' \xrightarrow{\alpha_1} (A \cap H') \oplus (A \cap K') \xrightarrow{\alpha_2} A \cap G' \to 0,\]
which induces an exact sequence 
\[0 \rightarrow \Hom(A \cap G',D) \xrightarrow{\alpha_2^*} \Hom(A \cap H',D) \oplus \Hom(A \cap K',D) \xrightarrow{\alpha_1^*} \Hom(Z,D) \to 0,\]
in which $\alpha_2^*$ is the homomorphism $(\res,\res)$.

The homomorphism $\alpha_1^*$ being surjective, for any  $f \in \Hom(Z,D)$, there exists $g \in \Hom(A \cap H',D) \oplus \Hom(A \cap K',D)$ such that $f=\alpha_1^*(g)$.  Let $g_1 \in \Hom(A \cap H',D) \oplus \Hom(A \cap K',D)$ be another element such that $f=\alpha_1^*(g_1)$. Then there exists $\nu \in \Hom(A \cap G',D)$ such that $g-g_1=\alpha_2^*(\nu)$. For the convenience of writing, set $\zeta = \inf \circ  \theta^{-1}$ (recall that $\theta$ is an isomorphism). Now, using the commutativity of the topmost rectangle of Diagram 1, we get
\begin{eqnarray*}
\zeta \circ (\tra,\tra,0)(g) &=& \zeta \circ (\tra,\tra,0)(g_1)+\zeta \circ (\tra,\tra,0)(\alpha_2^*(\nu))\\
&=& \zeta \circ (\tra,\tra,0)(g_1)+ \zeta \circ \theta \circ \tra(\nu)\\
&=&\zeta \circ (\tra,\tra,0)(g_1).
\end{eqnarray*}
Hence $\zeta \circ (tra,tra,0)$ is independent of the choice of $g \in \Hom(A \cap H',D) \oplus \Hom(A \cap K',D)$ with $\alpha_1^*(g)=f$. Setting $\chi(f) = \zeta \circ (\tra,\tra,0)$, we get a well defined map $\chi : \Hom(Z,D) \to \Ho^2(G,D)$.
It is now clear that $\chi$ is a homomorphism. 

\begin{thm}\label{thmnew}
The following sequence is exact:
\[0  \to \Hom(Z,D) \xrightarrow{\chi} \Ho^2(G,D) \xrightarrow{\theta'}  \Ho^2(H,D) \oplus \Ho^2(K,D) \oplus \Hom(H \otimes K, D).\]
\end{thm}
\begin{proof}
Suppose that  $f \in \Hom(Z,D)$ and $\chi(f)=0$. Then $\inf \circ \theta^{-1} \circ (\tra,\tra,0)(g)=0$ for some $g \in \Hom(A \cap H',D) \oplus \Hom(A \cap K',D)$ such that $f=\alpha_1^*(g)$. Thus  there exists $\eta \in \Hom(A \cap G',D)$ such that $\theta^{-1}\circ (\tra,\tra,0)(g)=\tra(\eta)$ by the commutativity of Diagram 1. Then $(\tra,\tra,0)(g)=\theta \circ \tra(\eta)=(\tra,\tra,0) \circ (\res,\res) (\eta)=(\tra,\tra,0) \circ \alpha_2^*(\eta)$. Since $(\tra,\tra,0)$ is a monomorphism, we have $g=\alpha_2^*(\eta)$. Thus $f=\alpha_1^* \circ \alpha_2^*(\eta)=0$, which, $f$ being an arbitrary element, proves that   $\chi$ is a monomorphism.
 That $\im(\chi) = \Ker (\theta')$ is now clear from Lemma \ref{lemma2}, and the proof is complete. $\hfill\square$

\end{proof}

The following is an immediate consequence of the preceding theorem.

\begin{cor}\label{corz1}
If $Z=1$, then 
$$\theta' : \Ho^2(G,D)\to \Ho^2(H,D) \oplus \Ho^2(K,D) \oplus \Hom(H \otimes K, D)$$
 is a monomorphism.
\end{cor}

Using the argument as in Theorem \ref{thmnew} and the observation that 
 $$\inf\big(\theta^{-1}(\Hom(H/A \otimes K/A,D))\big) \cap \ker(\theta') = \{0\},$$ 
 which follows from the commutativity of the middle rectangle of Diagram 1, we get: 
 
\begin{cor} 
The following sequence is exact
\[0 \to \Hom(Z,D) \oplus \Hom(H/A \otimes K/A,D) \xrightarrow{(\chi, \inf \circ \theta^{-1})} \Ho^2(G,D)  \xrightarrow{(\res,\res)}
\Ho^2(H,D) \oplus \Ho^2(K,D).\]
In particular,
$\Hom(Z,D) \oplus \Hom(H/A \otimes K/A,D)$ embeds in $\Ho^2(G,D)$.
\end{cor}

As we know by Theorem A that $\Hom(Z,D)$ embeds in $\Ho^2(G/Z,D)$. We now prove a much stronger result in the following
\begin{thm}\label{thmembd}
$\Hom(Z,D)$ embeds in $\Ho^2(H/A, D)/L \oplus \Ho^2(K/A, D)/M$, where $L \cong \Hom \big((A\cap H')/Z, D \big)$ and $M \cong \Hom \big( (A\cap K')/Z, D \big)$.
\end{thm}
\begin{proof}
Let $\alpha : \Hom(A \cap G', D) \to \Hom(Z, D)$ be the epimorphism induced by the inclusion $Z \hookrightarrow A \cap G'$. Set $Y_1 = \im\big(\inf : \Ho^2(G/A, D) \to \Ho^2(G/Z, D)\big)$.
Since $G/Z$ is a central product of $H/Z$ and $K/Z$ with $(H/Z)' \cap (K/Z)' = 1$,  it follows that
$Y_1$ is isomorphic to $\Ho^2(H/A, D)/L \oplus \Ho^2(K/A, D)/M \oplus \Hom(H/A \otimes K/A, D)$,
where $L \cong \Hom \big((A\cap H')/Z, D \big)$ and $M \cong \Hom \big( (A\cap K')/Z, D \big)$.

 Consider the following commutative  diagram (with rows not necessarily exact):
\[
 \xymatrix{
0 \;\;\ar[r] & \;\;\;\Hom(A \cap G', D)   \ar[r]^{\tra} \ar[d]^{\alpha}  \;\;\;& \;\;\Ho^2(G/A, D)\;\;\; \ar[d]^{\inf}\ar[r]^{\;\;\;\;\;\;\;\theta} &\;\;\;\;\; X\;\;\;\;\; \ar[d]^{(p_1, p_2, 1)}\\
0 \;\;\ar[r]  &\;\;\; \Hom(Z, D)  \ar[r]_{\tra}\;\;\; & \;\;Y_1 \;\;\;\ar[r]_{\;\;\;\;\;\;\;\;\bar{\theta} }  &\;\; Y, }\\
 \]
where 
$$X =\Ho^2(H/A, D) \oplus \Ho^2(K/A, D) \oplus \Hom(H/A \otimes K/A, D),$$
$$Y = \Ho^2(H/A, D)/L \oplus \Ho^2(K/A, D)/M \oplus \Hom(H/A \otimes K/A, D),$$
$\bar{\theta}$ is an isomorphism
and  $p_i$, $i =1, 2$, are  natural projections.

 Let $\beta \in \Hom(Z, D)$. Then there exists $\bar{\beta} \in \Hom(A \cap G', D)$ such that $\beta = \alpha(\bar{\beta})$. Let  $\tra(\bar{\beta}) = \xi \in \Ho^2(G/A, D)$. The element  $\xi$ is represented by a $2$-cocycle $f$  given by
\[f(\bar{x}, \bar{y}) = \bar{\beta}(\mu(\bar{x}) \mu(\bar{y}) \mu(\bar{x}\bar{y})^{-1}),  \;\; \bar{x} = xA \in G/A   \;\;  \text{and}   \;\; \bar{y} = yA \in G/A, \]
where $\mu$ represents the section $\mu: G/A \to G$  in  the exact sequence $1 \to A \to G \to G/A \to 1$.

Recall that $\theta = (\res, \res, \nu_1)$, where $\nu_1(\xi) : H/A \otimes K/A \to D$ for $\bar{h} = hA \in H/A$ and $\bar{k} = kA \in K/A$ is given by
\[\nu_1(\xi)(\bar{h}(H/A)', \bar{k}(K/A)') = f(\bar{h}, \bar{k}) - f(\bar{k}, \bar{h}).\]
Plugging in the value of  $f$ we have
\begin{eqnarray*}
\nu_1(\xi)(\bar{h}(H/A)', \bar{k}(K/A)') &=& f(\bar{h}, \bar{k}) - f(\bar{k}, \bar{h})\\
&=& \bar{\beta}\big(\mu(\bar{h}) \mu(\bar{k})  \mu(\bar{h}\bar{k})^{-1}\big)  - \bar{\beta}\big(\mu(\bar{k}) \mu(\bar{h}) \mu(\bar{k}\bar{h})^{-1}\big)\\
&=& \bar{\beta}\big(\mu(\bar{h}) \mu(\bar{k}) \mu(\bar{h} \bar{k})^{-1} \mu(\bar{k}\bar{h})  \mu(\bar{h})^{-1}\mu(\bar{k})^{-1} \big)\\
&=&  \bar{\beta}\big(\mu(\bar{h}) \mu(\bar{k}) \mu(\bar{h} \bar{k})^{-1} \mu(\bar{h}\bar{k})  \mu(\bar{k})^{-1}\mu(\bar{h})^{-1} \big)\\
&=& 0.
\end{eqnarray*}
Hence $\theta(\tra(\bar{\beta})) \in  \Ho^2(H/A, D)  \oplus \Ho^2(K/A, D)$. That  $\bar{\theta}(\tra(\beta)) \in  \Ho^2(H/A, D)/L \oplus \Ho^2(K/A, D)/M$ now follows by the commutativity of the above diagram, which completes  the proof. \hfill $\Box$

\end{proof}

Using an argument similar to one as in the preceding proof, we can also prove

\begin{thm}\label{thmre}
$\Hom(Z,D)$ embeds in $\Ho^2(H/Z, D) \oplus \Ho^2(K/Z, D)$.
\end{thm}

The following is now an immediate consequence of Theorem A and the preceding theorem.
\begin{cor}
If $A = Z$, then 
$$\Ho^2(G,D) \cong \big(\Ho^2(H/Z,D) \oplus \Ho^2(K/Z,D) \big)/\Hom(Z,D) \oplus \Hom(H/Z \otimes K/Z, D).$$
\end{cor}
\vspace{.2in}

We are now ready to prove Theorem B.

\vspace{.2in}
\noindent {\it Proof of Theorem B.}
As already observed in the proof of Theorem \ref{thmembd},  $\im(\inf : \Ho^2(G/A, D) \to \Ho^2(G/Z, D))$ is isomorphic to 
 $\Ho^2(H/A, D)/L \oplus \Ho^2(K/A, D)/M \oplus \Hom(H/A \otimes K/A, D)$. The first assertion now follows from  Theorem A and Theorem \ref{thmembd}.

The second assertion follows from Theorem A, Corollary \ref{corz1} (with $G$ replaced by $G/Z$) and Theorem \ref{thmre}.$\hfill\square$\\

It is perhaps an appropriate place to remark  that
$$\Ho^2(G,D)\cong (\Ho^2(H/A, D)/L \oplus \Ho^2(K/A, D)/M )/N  \oplus \Hom(H/A \otimes K/A, D)$$ if and only if $\inf: \Ho^2(G/A,D) \to \Ho^2(G/Z,D)$ is an epimorphism, where $L, M$ and $N$ are as defined above.

\vspace{.1in}

The following result is immediate from the commutativity of the bottommost part of Diagram 1.
\begin{prop}\label{lastlemma}
Let $\xi \in \Ho^2(\bar{G},D)$ such that  $\theta'(\xi)=(\xi_1,\xi_2,t)$, where $\bar{G} = G/Z$. Further, let either  $\res^{\bar{H}}_{\bar{A}}(\xi_1) = 0$ or  $\res^{\bar{K}}_{\bar{A}}(\xi_2) =0$. Then the following statements are equivalent:

(i)~$\xi \in \im\big(\inf:\Ho^2(\bar{G}/\bar{A},D) \to \Ho^2(\bar{G},D)\big)$;

(ii)~$\mu^*(t)=0$;

(iii)~$\psi(\xi_1)=\psi(\xi_2)=0$.
\end{prop} 

As a consequence we have

\begin{cor} 
If $\inf: \Ho^2(H/A,D) \to \Ho^2(H/Z,D)$ and $\inf: \Ho^2(K/A,D) \to \Ho^2(K/Z,D)$ are epimorphisms, then 
\[H^2(G,D) \cong \big(H^2(H/Z,D) \oplus H^2(K/Z,D) \big)/ \Hom(Z,D) \oplus \Hom(H/A \otimes K/A, D).\]
More precisely, the first embedding in Theorem B is an isomorphism.
\end{cor}

Using Proposition \ref{lastlemma}, we also have
\begin{thm}
If the second embedding in Theorem B is an isomorphism, then so is the first.
\end{thm}
\begin{proof}
Since the isomorphism
$$\Ho^2(G, D)  \cong \big(\Ho^2(H/Z, D) \oplus \Ho^2(K/Z,D)\big)/\Hom(Z,D) \oplus \Hom(H/Z \otimes K/Z, D)$$
is induced by the monomorphism $\theta'$ as defined in \eqref{eqn1} with  $G$ replaced by $G/Z$, it follows from the commutative diagram 
\[
 \xymatrix{
    0     \ar[d]&    &0 \ar[d]\\
   \Hom(Z, D)   \ar[d]_{\tra}  & \xrightarrow{(\id,\id)}  &\Hom(Z, D) \oplus \Hom(Z, D) \ar[d]^{(\tra, \tra, 0)}\\
    \Ho^2(G/Z,D)   &  \xrightarrow{\theta'} &\Ho^2(H/Z,D) \oplus \Ho^2(K/Z,D) \oplus \Hom(H/Z \otimes K/Z,D) .}
\]
that $\theta'$ is  an isomorphism.
Let  $t \in \Hom(H/Z \otimes K/Z, D)$. Then there exists $\xi \in \Ho^2(G/Z, D)$ such that  $\theta'(\xi)=(0,0,t)$. It then follows from Diagram 1 (for $G/Z$ in place of $G$) that  $(\res,\psi)(\xi)=0$. By  Proposition \ref{lastlemma} we then have $\mu^*(t)=0$, which shows that $\lambda^* :  \Hom(H/A \otimes K/A,D) \to \Hom(H/Z \otimes K/Z,D)$ is an epimorphism, and hence, an isomorphism. 
Similarly,  considering the elements $\xi_1 \in \Ho^2(H/Z, D)$ and  $\xi_2 \in \Ho^2(K/Z,D)$ in succession, the above argument also shows that $\inf :  \Ho^2(H/A,D) \to \Ho^2(H/Z,D)$ and $\inf : \Ho^2(K/A,D)\Ho^2(K/Z,D)$ are epimorphisms. Hence  $ \Ho^2(H/A, $ $D)/L \cong \Ho^2(H/Z, D)$ and $\Ho^2(K/A, D)/M \cong \Ho^2(K/Z,D)$. It now follows that the first embedding in Theorem B is an isomorphism. \hfill $\Box$

\end{proof}
\vspace{.2in}

We conclude this section with the following remark made by J.  Wiegold while reviewing \cite{EHS} for AMS (see MR0349854 (50 \#2347)). Let $G$ be the direct product $G = H \times K$ of its normal subgroups $H$ and $K$, and $U$ be an arbitrary central subgroup of $G$.  Then  $G/U$  can be viewed as a central product of $HU/U$ and $KU/U$. Thus all the above results make sense for  $\Ho^2(G/U, D)$.

\section{Examples}

In this section we present several examples (all for finite $p$-groups) to show that various situations of  Theorem B can indeed occur. By $\mathbb{Z}_p^{(n)}$, we denote  the elementary abelian $p$-group of rank $n$, where $n \ge 1$. We start with an example which shows that neither of the two embeddings of Theorem B is necessarily an isomorphism.

\vspace{.1in}

\noindent {\bf Example 1.}  Let $H$ be  the extraspecial $p$-groups of order $p^3$ and exponent $p$ and $K=\mathbb{Z}_p^{(n+1)}$, where $n \ge 1$. Let $G$ be a central product of $H$ and $K$ amalgamated at $A \cong H' \cong \mathbb{Z}_p$. Note that  $G=H \times \mathbb{Z}_p^{(n)}$. It is easy to see that 
$$\M(G) \cong \mathbb{Z}_p^{\big(\frac{1}{2}n(n+3) +2\big)}.$$
 Note that $Z = H' \cap K' =1$. Then 
\[\M(H/A)/ \Hom(A \cap H',\mathbb{C}^{*}) \oplus \M(K/A)/ \Hom(A \cap K',\mathbb{C}^{*}) \oplus \Hom(H/A \otimes K/A,\mathbb{C}^{*})\] is isomorphic to  $\mathbb{Z}_p^{\big(\frac{1}{2}n(n+3)\big)}$, which is   strictly contained in $M(G)$.  Since 
$$\M(H) \oplus \M(K) \oplus \Hom(H \otimes K, \mathbb{C}^{*}) \cong \mathbb{Z}_p^{(\frac{1}{2}(n+1)(n+4)+2)},$$ it properly contains  $M(G)$.

\vspace{.1in}

Before  proceeding further, we mention the following interesting result by M. R. Jones \cite[Theorem 4.1(i)]{MR}.

\begin{thm}\label{jonesthm}
Let X be a finite group and N a central subgroup. Then
\[|\M(X)|  |X' \cap N| \;\; \text{divides} \;\;  |\M(X/N)| |\M(N)|  |(X/N) \otimes N|.\]
\end{thm}

The following two examples show that  the first  embedding in Theorem B can very well be an isomorphisms, but the second one can still be strict (i.e., not an isomorphism).

\vspace{.2in}

\noindent{\bf Example 2.} Consider the group $G$ presented as
$$G=\langle \alpha, \alpha_1, \alpha_2, \gamma \mid [\alpha_1,\alpha]=\gamma^{p^2}=\alpha_2, \alpha^{p}=\alpha_1^p=\alpha_2^p=1\rangle.$$
 Take $H=\langle \alpha, \alpha_1, \alpha_2 \mid [\alpha_1,\alpha]=\alpha_2, \alpha^{p}=\alpha_1^p=\alpha_2^p=1\rangle$
 and $K=\langle \gamma \rangle \cong \mathbb{Z}_{p^3}$. Then  $G$ is a central product of $H$ and $K$ amalgamated at $A \cong \gen{\alpha_2} \cong \gen{\gamma^{p^2}}$. Note that $Z = 1$ and 
 \[\M(H/A)/ \Hom(A \cap H',\mathbb{C}^{*}) \oplus \M(K/A)/ \Hom(A \cap K',\mathbb{C}^{*}) \oplus \Hom(H/A \otimes K/A,\mathbb{C}^{*})\]
 is isomorphic to $\mathbb{Z}_p^{(2)}$.   By Theorem \ref{jonesthm},  we have $|\M(G)| \leq p^2$.  Hence $\M(G) \cong  \mathbb{Z}_p^{(2)}$, and therefore  the first embedding in Theorem B is  an isomorphisms. It is easy to see that 
 \[\M(H) \oplus \M(K) \oplus \Hom(H \otimes K,\mathbb{C}^{*}) \cong  \mathbb{Z}_p^{(4)},\]
which shows that the second embedding is strict.
\vspace{.1in}

\noindent{\bf Example 3.} Consider the group $G$ presented as
$$G=\langle \alpha,\alpha_1,\alpha_2,\alpha_3,\gamma \mid [\alpha_1,\alpha]=\alpha_2,[\alpha_2,\alpha]=\gamma^p=\alpha_3,\alpha^p=\alpha_i^{(p)}=1, i=1,2,3 \rangle.$$ 
Take $H=\langle \alpha,\alpha_1,\alpha_2,\alpha_3 \mid [\alpha_1,\alpha]=\alpha_2,[\alpha_2,\alpha]=\alpha_3,\alpha^p=\alpha_i^{(p)}=1, i=1,2,3 \rangle$ and $K=\langle \gamma \rangle \cong \mathbb{Z}_{p^2}$.
Then $G$ is a central product of $H$ and $K$ amalgamated at $A \cong \gen{\alpha_3} \cong \gen{\gamma^{p}}$ and  $Z = 1$.
Note that 
\[M(H/A)/\Hom(A\cap H', \mathbb{C}^{*}) \oplus M(K/A)/ \Hom(A\cap K', \mathbb{C}^{*}) \oplus \Hom(H/A \otimes K/A, \mathbb{C}^{*})\]
is isomorphic to $\mathbb{Z}_p^{(3)}$,
which embeds in  $\M(G)$. Again by  Theorem \ref{jonesthm}, we have $|\M(G)| \leq p^3$. Hence the first embedding is an isomorphism. That the second one is not can be easily seen as in Example 2.

\vspace{.1in}

We finally present an example which shows that both the embeddings in Theorem B can be isomorphisms.

\vspace{.1in}

\noindent{\bf Example 4.}
Let $H$ be the extraspecial $p$-groups of order $p^3$ and exponent $p^2$ and $K \cong \mathbb{Z}_{p^{n+1}}$, the cyclic group of order $p^{n+1}$, where $n \ge 1$.
Let $G$ be a central product of $H$ and $K$ amalgamated at $A \cong H' \cong \mathbb{Z}_p$. Note that  $G=H \times \mathbb{Z}_{p^{n}}$. It is easy to see that 
$\M(G) \cong \mathbb{Z}_p^{(2)}.$
 Note that $Z = H' \cap K' =1$. Then 
\[\M(H/A)/ \Hom(A \cap H',\mathbb{C}^{*}) \oplus \M(K/A)/ \Hom(A \cap K',\mathbb{C}^{*}) \oplus \Hom(H/A \otimes K/A,\mathbb{C}^{*})\] is isomorpic to  $\mathbb{Z}_p^{(2)}$. Also
$$\M(H) \oplus \M(K) \oplus \Hom(H \otimes K, \mathbb{C}^{*}) \cong \mathbb{Z}_p^{(2)}.$$ 
Hence both the embeddings are isomorphisms.

\end{document}